\documentclass[preprint,12pt,a4paper]{elsarticle}
\usepackage[T1]{fontenc}
\usepackage{lmodern,amsmath,amssymb,amsthm,mathtools,mathrsfs}
\usepackage{microtype,booktabs,array,enumitem}
\usepackage{tikz}
\usetikzlibrary{arrows.meta,positioning}
\usepackage[hidelinks]{hyperref}
\usepackage[nameinlink,noabbrev]{cleveref}
\numberwithin{equation}{section}
\newtheorem{theorem}{Theorem}[section]
\newtheorem{proposition}[theorem]{Proposition}
\newtheorem{lemma}[theorem]{Lemma}
\newtheorem{corollary}[theorem]{Corollary}
\theoremstyle{definition}

\theoremstyle{remark}
\newtheorem{remark}[theorem]{Remark}

\crefname{theorem}{Theorem}{Theorems}
\Crefname{theorem}{Theorem}{Theorems}
\crefname{proposition}{Proposition}{Propositions}
\Crefname{proposition}{Proposition}{Propositions}
\crefname{lemma}{Lemma}{Lemmas}
\Crefname{lemma}{Lemma}{Lemmas}
\crefname{corollary}{Corollary}{Corollaries}
\Crefname{corollary}{Corollary}{Corollaries}
\crefname{definition}{Definition}{Definitions}
\Crefname{definition}{Definition}{Definitions}
\crefname{remark}{Remark}{Remarks}
\Crefname{remark}{Remark}{Remarks}
\crefname{section}{Section}{Sections}
\Crefname{section}{Section}{Sections}

\crefname{figure}{Figure}{Figures}
\Crefname{figure}{Figure}{Figures}

\newcommand{\Z}{\mathbb Z}
\newcommand{\Q}{\mathbb Q}
\newcommand{\N}{\mathbb Z_{\geq0}}
\newcommand{\LL}{\mathsf L}
\newcommand{\RR}{\mathsf R}
\newcommand{\PP}{\mathcal P}
\newcommand{\MM}{\mathcal M}
\newcommand{\GG}{\mathcal G}
\newcommand{\TT}{\mathcal T}
\newcommand{\OO}{\mathcal O}
\newcommand{\CQ}{\mathscr C_Q}
\newcommand{\Ih}{\widehat I}
\newcommand{\IhQ}{\widehat I_Q}
\newcommand{\ct}{\widetilde C}
\newcommand{\om}{\omega_Q}

\DeclareMathOperator{\Ad}{Ad}

\DeclareMathOperator{\Spec}{Spec}

\DeclareMathOperator{\wt}{wt}

\journal{}

\makeatletter
\def\ps@pprintTitle{%
  \let\@oddhead\@empty
  \let\@evenhead\@empty
  \let\@oddfoot\@empty
  \let\@evenfoot\@empty}
\makeatother

\biboptions{sort&compress}
\hypersetup{pdftitle={Spectral Rigidity in Toroidal Grothendieck Rings and Multiparameter Quantum Groups},pdfauthor={Xiaomin Tang and Yu Zhang}}

\begin{document}
\begin{frontmatter}

\title{Spectral Rigidity in Toroidal Grothendieck Rings and Multiparameter Quantum Groups}

\author[hlju]{Xiaomin Tang}
\ead{tangxm@hlju.edu.cn}
\author[hlju]{Yu Zhang\corref{cor1}}
\ead{zhangyu@hlju.edu.cn}
\cortext[cor1]{Corresponding author.}

\address[hlju]{School of Mathematics and Statistics, Heilongjiang University,
Harbin 150080, P. R. China}

\begin{abstract}
Let \(Q\) be a bipartite Dynkin quiver of finite simply-laced type and
\(\mathscr C_Q\) the corresponding Hernandez--Leclerc category.  We study
a two-parameter specialization of the Fedele--Hernandez toroidal
Grothendieck ring and identify its generic positive part with a
Cartan-type multiparameter quantum group.

The structural input is the identity
\[
2D^{(2)}-D_0=\omega_Q\circ(\deg_Q,\deg_Q),
\]
where \(D_0,D^{(2)}\) are the two effective toroidal commutation forms and
\(\omega_Q\) is the antisymmetric Euler form of \(Q\).  It realizes the
specialized character algebra, after an auxiliary scalar extension, as a
graded bicharacter twist of the one-parameter algebra at
\(v=t_0t^{1/2}\).  This yields the required root grading and PBW graded
dimensions.

For every source--sink edge \(r\to s\), these dimensions combine with the
local \(A\)--\(Y\) commutation formula to give a quadratic annihilator for
\(\operatorname{Ad}_{Z_r}\).  Consequently
\[
\operatorname{Spec}_{Z_r}(Z_s)=\{t_0t,t_0^{-1}\},
\]
without thinness or type-specific character formulas.  The two spectral
values yield the oriented multiparameter Serre relations.  With
\[
q_{ii}=t_0^2t,\qquad
q_{ij}=t_0^{-1},\qquad
q_{ji}=t_0^{-1}t^{-1}\quad(i\to j),
\]
and \(q_{ij}=1\) on nonedges, we obtain
\[
U_{\mathbf q}^{+}(\mathfrak g)
\simeq
\mathscr K_\infty(\mathscr C_Q)\otimes_RK
\]
for \(\mathfrak g\) of type \(A_n,D_n,E_6,E_7,E_8\).
\end{abstract}


\end{frontmatter}

\section{Introduction}
The representation theory of quantum affine algebras provides several
noncommutative deformations of Grothendieck rings. Frenkel--Reshetikhin
\(q\)-characters record the joint generalized eigenspaces of the Drinfeld
Cartan currents \cite{FR}. Their deformations, constructed geometrically by
Nakajima and algebraically by Hernandez, lead to quantum Grothendieck rings
and algorithms for characters of simple modules \cite{Nak,H}. For a Dynkin
quiver \(Q\) of type \(A,D\), or \(E\), Hernandez and Leclerc introduced a
monoidal category \(\CQ\) and identified its quantum Grothendieck ring with
the positive part of the quantum group of the same type
\cite[Theorem~6.1]{HL}. Their
construction relates spectral positions, Auslander--Reiten translation,
inverse quantum Cartan coefficients, and dual PBW generators.

Fedele and Hernandez introduced a deformation involving several toroidal
parameters and constructed its flat toroidal character theory
\cite[Theorem~5.8]{FH}. Flatness requires relations between these parameters;
it does not hold for the naive unrestricted deformation. Their quotient
construction has a particularly rigid local feature: a root monomial
\(A_{i,p}\) quasi-commutes with a spectral variable \(Y_{j,s}\) through a
single combination of the toroidal parameters, and the exponent vanishes
unless \(i=j\) and \(|p-s|=1\). They exhibited a multiparameter positive-part
phenomenon in type \(A_2\) and proposed a wider relationship for the categories
\(\CQ\) \cite[\S8.3]{FH}.

We give a uniform algebraic treatment of the two-parameter specialization
for bipartite quivers of finite simply-laced type. The central question is
local: If a source-side simple-root class is
a monomial, how many eigenvalues can its conjugation have on an adjacent
distinguished character? In type \(A\), column tableaux can make cancellations
visible. At a trivalent vertex, and for non-thin fundamental modules, a
support-by-support approach is less natural. We instead use a grading and a
small homogeneous dimension.

A preliminary issue is essential. Ambient flatness alone does not establish
the root grading needed to compare cubic words, and coefficients in a
toroidal triangular basis cannot be identified a priori with ordinary
tensor-product multiplicities. We resolve both issues by comparing the
toroidal product directly with the one-parameter product before invoking
rank-two dimensions.

\subsection*{Standing hypotheses}
Throughout the paper, \(Q\) is a connected bipartite Dynkin quiver of
finite simply-laced type, with vertex set \(I=I_+\sqcup I_-\) and every
edge oriented from \(I_+\) to \(I_-\).  The corresponding simple Lie
algebra is denoted by \(\mathfrak g\).  All deformation parameters are
generic: \(t_0^{1/2}\) and \(u=t^{1/2}\) are algebraically independent
over \(\mathbb Q\).  The coefficient fields and local rings used later are
obtained from these parameters by the explicitly stated scalar extensions
and localizations.  No root-of-unity specialization is considered.

\subsection*{The kernel identity and the main result}
Let \(I=I_+\sqcup I_-\), and orient every edge of the Dynkin diagram from
\(I_+\) to \(I_-\). Normalize the height function by \(\xi_i=1\) on
\(I_+\) and \(\xi_i=0\) on \(I_-\). Let
\(\LL=\bigoplus_{i\in I}\Z\alpha_i\) be the finite root lattice. We use the
alternating form
\begin{equation}\label{eq:intro-euler}
 \om(\alpha_i,\alpha_j)=
 \begin{cases}1&i\longrightarrow j,\\-1&j\longrightarrow i,\\0&\text{otherwise}.
 \end{cases}
\end{equation}
Equivalently, \(\om(\alpha,\beta)=\langle\beta,\alpha\rangle_Q-
\langle\alpha,\beta\rangle_Q\). The sign convention is fixed by this formula.

Put \(F=\Q(t_0^{1/2})\), \(u=t^{1/2}\), \(R=F[u^{\pm1}]\), and
\(K=F(u)\). We specialize every positive even toroidal parameter to \(1\)
and every negative even parameter to \(t^{-1}\), keeping \(t_0\). The
resulting skew forms are denoted by \(D_0,D^{(2)}\). We work with the algebra
\(\GG_R(Q)\) generated by the specialized truncated fundamental characters
inside their based torus; this is the embedded realization of the
specialized toroidal Grothendieck ring used throughout the paper.

\begin{theorem}[Kernel decomposition]\label{thm:intro-kernel}
There is a surjective root-degree map from the spectral monomial lattice to
\(\LL\), with kernel generated by the root monomials, such that
\begin{equation}\label{eq:intro-kernel}
 2D^{(2)}(m,n)-D_0(m,n)=\om(\deg_Qm,\deg_Qn).
\end{equation}
Consequently, after adjoining \(w\) with \(w^2=u\), the toroidal product is
the graded twist of the one-parameter product at \(v=t_0u\) given by
\[
 x*y=w^{\om(\alpha,\beta)}(x*_v y)
 \qquad(\deg_Qx=\alpha,\ \deg_Qy=\beta).
\]
The same statement holds for the subalgebras generated by the truncated
fundamental characters of \(\CQ\).
\end{theorem}
The auxiliary extension is used only to express the symmetric, based-monomial
normalization. The final realization is over \(K\), not merely over \(K(w)\).
The character-generated algebra has a well-defined positive root grading,
and its homogeneous dimensions are the Kostant partition numbers. These are
consequences of \cref{thm:intro-kernel}, not assumptions about a putative
categorification of the toroidal coefficients.

For each simple root let \(Z_i\) be the specialized truncated fundamental
class indexed by \(\varphi^{-1}(\alpha_i,0)\). Define the Cartan-type
multiparameter matrix by
\begin{equation}\label{eq:intro-q}
 q_{ii}=t_0^2t,\qquad
 q_{ij}=t_0^{-1},\quad q_{ji}=t_0^{-1}t^{-1}\quad(i\to j),
 \qquad q_{ij}=1\quad(i\ne j,\ i\not\sim j).
\end{equation}
We use the positive-part presentation of Pei--Hu--Rosso \cite{HPR}.
For complementary cocycle-deformation and specialization viewpoints on
multiparameter quantum groups, see \cite{Garcia,GG2026}.
\begin{theorem}[Generic simply-laced realization]\label{thm:intro-main}
For every connected bipartite Dynkin quiver of type
\(A_n,D_n,E_6,E_7\), or \(E_8\), the assignment \(E_i\mapsto Z_i\) induces
an isomorphism
\[
 U_{\mathbf q}^{+}(\mathfrak g)\ \xrightarrow{\ \sim\ }\
 \GG_R(Q)\otimes_RK.
\]
For each source--sink edge \(r\to s\), the local conjugation spectrum is
\[
\Spec_{Z_r}(Z_s)=\{t_0t,t_0^{-1}\}.
\]
\end{theorem}
The rank-one case is included by taking \(q_{11}=t_0^2t\); its positive part
is a polynomial algebra and there is no adjacent-pair assertion.

\subsection*{Method and relationship with cluster structures}
The proof uses a local spectral argument.  Suppose the three cubic words
\(X^2Z\), \(XZX\), and \(ZX^2\) lie in a two-dimensional space and
\(X\) is invertible in the ambient torus.  Then a polynomial of degree at
most two annihilates \(Z\) under \(\Ad_X\).  Monomial eigenvectors make
the spectral bound explicit, and the one-parameter fiber identifies its two
roots.  The reverse Serre relation follows from a second two-dimensional
homogeneous component.  Its proof separates the eigenvalues
\(a^2\), \(ab\), and \(b^2\), without assuming a relation between the
two mixed products of the eigenblocks.

The key extra ingredient is the passage from local \(A\)--\(Y\) rigidity to
the global degree form \(\om\). This supplies the missing grading input and
also determines the apparently free nonadjacent parameters: they are \(1\)
for this bipartite specialization. We give an explicit local bicharacter
construction for the positive part and check the direction of its cocycle
ratio on the three Serre words.

Cluster structures remain part of the setting, rather than an additional
unproved hypothesis. We prove a compatibility identity for homogeneous seeds:
if \(G\widetilde B=0\), then adding the degree-form correction does not
change the exchange-column pairing. Thus a one-parameter compatibility
constant \(2\) yields a second constant \(1\). This result is stated in
terms of homogeneous seed data; we do not need a separate identification of
all toroidal cluster variables with deformed simple classes.

The broader quantum-cluster setting has been developed in
\cite{HLcluster,Bittmann,BittmannJLMS,GLS}.  Quantum Grothendieck-ring
isomorphisms, positivity, and their cluster-theoretic consequences are
studied in \cite{HO,FHOO}; quantum virtual Grothendieck rings and
skew-symmetrizable quantum cluster structures are developed in
\cite{JLO}.  Related quantum cluster structures on quantized coordinate
rings, including the quantum coordinate ring of skew-symmetric matrices,
are studied in \cite{ZT2026}.  Recent work continues to connect quantum
Grothendieck rings with monoidal categorification, deformations, and
shifted quantum affine algebras \cite{DS,FHJan,Paganelli}. These results concern different categories or
quantizations; the present question is the specific Fedele--Hernandez
specialization and its finite-type positive-part realization.  In
particular, our argument does not deduce the realization from a cluster
structure.  It derives the multiparameter Cartan data from the toroidal
commutation forms and the rank-two graded dimensions, and only afterwards
compares the resulting presentation with the multiparameter quantum
group. The quotient
of the spectral weight lattice by root monomials also has an established
representation-theoretic role \cite{CM,FH}; here we use its explicit
bipartite presentation to compute the residual commutation form.

Two structural statements drive the proof.  The Euler-correction identity
identifies the residual toroidal skew direction with the antisymmetric
Euler form of \(Q\).  The local spectral theorem then converts the
rank-two PBW dimension into a quadratic constraint on the adjoint action.
The remaining steps reconstruct the Cartan parameters and globalize the
local relations by a cocycle/PBW specialization argument.  This separates
the local rigidity mechanism from the global presentation.

For ease of reference, \(\GG_R(Q)\) denotes the specialized
character-generated algebra inside the based torus, \(Z_i\) its
distinguished simple-root generators, \(\LL\) the finite root lattice,
and \(\omega_Q\) the antisymmetric Euler form fixed by
\eqref{eq:intro-euler}.  The symbol \(\operatorname{Ad}_X\) always means
conjugation in the ambient based quantum torus; inverse monomials are not
asserted to belong to \(\GG_R(Q)\).

The paper is organized as follows. \Cref{sec:data} fixes the coefficient and
character conventions. \Cref{sec:kernel} proves the lattice quotient and
kernel decomposition. \Cref{sec:twist} establishes the character comparison
and graded dimensions. \Cref{sec:local} proves the spectral and Serre
statements. \Cref{sec:global} constructs the PBW form and proves the generic
isomorphism. Homogeneous seed compatibility and explicit branching examples
are treated in \cref{sec:cluster,sec:examples}.

\section{Spectral lattices and toroidal character algebras}\label{sec:data}
\subsection{Repetition-quiver conventions}
Fix a connected finite simply-laced Dynkin diagram and a bipartite
orientation \(Q\). We distinguish the quiver \(Q\) from its root lattice
\(\LL\). Let \(H\) be the adjacency matrix and \(C=2I-H\) the Cartan
matrix. We write
\[
 \Ih=\{(i,p)\in I\times\Z:p-\xi_i\in2\Z\}.
\]
The additive free abelian lattice with basis \(e_{i,p}\), \((i,p)\in\Ih\),
is denoted by \(\PP\). Its multiplicative notation is the lattice of finite
Laurent monomials in the variables \(Y_{i,p}\).

Let \(s_+=\prod_{i\in I_+}s_i\), \(s_-=\prod_{j\in I_-}s_j\), and
\(\tau=s_+s_-\). Products inside either color commute. With the
Hernandez--Leclerc convention \cite[\S2]{HL},
\begin{equation}\label{eq:gamma}
 \gamma_i=\begin{cases}
 \alpha_i&i\in I_+,\\
 \alpha_i+\displaystyle\sum_{j\sim i}\alpha_j&i\in I_-.
 \end{cases}
\end{equation}
These are the dimension vectors of the indecomposable injectives. They
form a \(\Z\)-basis of \(\LL\). The signed root labeling is
\begin{equation}\label{eq:degree-candidate}
 \beta_{i,p}=\tau^{(\xi_i-p)/2}\gamma_i.
\end{equation}
If \(\varphi(i,p)=(\beta,m)\), then
\(\beta_{i,p}=(-1)^m\beta\). In particular,
\(\IhQ=\varphi^{-1}(\Delta_+\times\{0\})\) is a finite set with
\(|\Delta_+|\) elements. Its column in direction \(i\) starts at
\((i,\xi_i)\) and proceeds downward until the first exit from level zero.
These are the conventions of \cite[\S\S2.2--2.3]{HL}.

The monoidal category \(\CQ\) consists of the finite-dimensional quantum
affine representations whose simple constituents have dominant highest
\(\ell\)-weight monomials supported on \(\IhQ\). The quantum affine
parameter in \(q\)-characters is fixed and is not one of the formal
deformation parameters \(t_0,t\).

\subsection{Root monomials and the inverse quantum Cartan matrix}
For \(r-\xi_i\in1+2\Z\), set
\begin{equation}\label{eq:root-monomial}
 A_{i,r}=Y_{i,r-1}Y_{i,r+1}\prod_{j\sim i}Y_{j,r}^{-1},\qquad
 a_{i,r}=e_{i,r-1}+e_{i,r+1}-\sum_{j\sim i}e_{j,r}.
\end{equation}
Let \(\RR\subset\PP\) be the subgroup generated by the \(a_{i,r}\).
The matrix \(C(z)=(z+z^{-1})I-H\) has inverse
\[
 \ct_{ij}(z)=\sum_{m\ge1}c_{ij}(m)z^m.
\]
Throughout, \(c_{ij}(m)=0\) for \(m\le0\). Symmetry and the recurrence are
\begin{equation}\label{eq:c-recur}
 \begin{aligned}
 c_{ij}(m)&=c_{ji}(m),\qquad c_{ij}(1)=\delta_{ij},\\
 c_{ij}(m+1)+c_{ij}(m-1)-\sum_{k\sim i}c_{kj}(m)&=0
 \qquad(m\ge1).
 \end{aligned}
\end{equation}
In particular \(c(2)=H\). With the zero extension the same Laplacian has
a single boundary term at \(m=0\), namely \(\delta_{ij}\).

In the sign convention of \cite[(11)--(12)]{FH}, the toroidal kernels are
\begin{align}\label{eq:Na}
 \mathcal N_a(i,p;j,s)
 ={}&c_{ij}(p-s-1+a)-c_{ij}(s-p-1+a)\notag\\
 &-c_{ij}(p-s+1+a)+c_{ij}(s-p+1+a).
\end{align}
They are skew-symmetric. If \(x=\sum x_{i,p}e_{i,p}\) and
\(y=\sum y_{j,s}e_{j,s}\), define
\[
 D_a(x,y)=\sum x_{i,p}y_{j,s}\mathcal N_a(i,p;j,s),\qquad
 D^{(2)}=-\sum_{h\ge1}D_{-2h}.
\]
The last sum is finite on any fixed pair: the kernels vanish for
\(a<-|p-s|\). The parity of \(c_{ij}(m)\), as in
\cite[Proposition 2.1]{HL}, implies that odd-indexed toroidal parameters do
not enter products in this parity component.

\subsection{The specialization and based product}
The Fedele--Hernandez parameter quotient imposes, for \(k\ge1\),
\begin{equation}\label{eq:FHrelations}
 t_{-2k-2}^{-1/2}t_{-2k}^{1/2}t_{2k}^{1/2}t_{2k+2}^{-1/2}=1.
\end{equation}
We use the coefficient homomorphism
\begin{equation}\label{eq:specialization}
 t_0^{1/2}\mapsto t_0^{1/2},\qquad
 t_{2h}^{1/2}\mapsto1,\qquad
 t_{-2h}^{1/2}\mapsto u^{-1}\quad(h\ge1).
\end{equation}
Every relation \eqref{eq:FHrelations} maps to \(1\). It is essential that
\emph{all} negative even parameters have the same image; retaining only
\(t_{-2}\) and setting the others to \(1\) is not this specialization.

The based torus \(\TT_R\) is free over \(R\) with basis \(\MM\), and its
multiplication is
\begin{equation}\label{eq:based-product}
 m*n=t_0^{D_0(m,n)/2}u^{D^{(2)}(m,n)}mn.
\end{equation}
The symbol \(mn\) on the right is a commutative Laurent monomial. In all
algebra expressions \(X^2\), \(XYZ\), and inverses refer to the associative
based product unless explicitly indicated otherwise. The commutation scalar
is \(t_0^{D_0(m,n)}t^{D^{(2)}(m,n)}\).

\subsection{Character subalgebras and lower support}
Write \(F_\infty(m_+)\) for the unique-dominant toroidal element of
\cite[Theorem 5.8]{FH}, and \(F_v(m_+)\) for its one-parameter analogue.
The former theorem concerns the intersection of the toroidal screening
subalgebras \cite[\S5.5, Theorem 5.8]{FH}. Their generators in direction
\(i\) are the variables
\(Y_{j,p}^{\pm1}\) for \(j\ne i\), together with
\begin{equation}\label{eq:screening-generator}
 B_{i,p}=Y_{i,p}+Y_{i,p}A_{i,p+1}^{-1}.
\end{equation}
The same expressions generate the one-parameter screening subalgebras.
We use both their intersection description and the unique-dominant
characterization; see \cite[\S5]{FH} and \cite{H}.

\begin{lemma}[Lower support]\label{lem:lower}
Each monomial in \(F_\infty(m_+)\), and in its specialization or
truncation, has the form \(m_+\prod A_{i,r}^{-v_{i,r}}\) with
\(v_{i,r}\in\N\) of finite support. The analogous assertion holds for
\(F_v(m_+)\).
\end{lemma}
\begin{proof}
In the recursive construction, an \(i\)-expansion at a monomial \(m\)
is assembled from factors \eqref{eq:screening-generator}. Each term has
exponent vector equal to that of \(m\) minus a nonnegative sum of
\(a_{i,r}\). Ordered products change only scalar coefficients. Thus the
cone \(m_+\prod A_{i,r}^{-\N}\) is stable under every expansion at a
monomial already in the cone. The correction steps take linear combinations
of such expansions and cannot introduce a new exponent vector outside it.
Induction over the finite sequence of steps producing a given term proves
the assertion for the recursive construction of \cite[Theorem 5.8]{FH}.
Specialization may cancel terms, and truncation deletes terms; neither
operation creates an exponent vector. The one-parameter proof is identical.
\end{proof}

Let \(\operatorname{tr}_Q\) denote the character truncation of
\cite[\S5.6]{FH}. Define \(\GG_R(Q)\subset\TT_R\) to be the algebra
generated by
\begin{equation}\label{eq:define-G}
 Z_{i,p}:=\operatorname{tr}_Q\left(F_\infty(Y_{i,p})
       \big|_{\eqref{eq:specialization}}\right),\qquad(i,p)\in\IhQ.
\end{equation}
Only variables in \(\IhQ\) occur in this embedded algebra. This definition
fixes the specialization convention independently of any proposed integral
presentation. We abbreviate \(Z_i=Z_{\kappa_i,p_i}\) when
\(\varphi(\kappa_i,p_i)=(\alpha_i,0)\). It remains to prove the grading
and dimension statements needed for the spectral argument.

\section{The spectral lattice and the Euler correction}\label{sec:kernel}
\subsection{The negative-even sum}
The zero extension of the inverse Cartan coefficients is important in the
following identity. No restriction to a boundary pair is needed.
\begin{lemma}[Finite telescoping]\label{lem:telescope}
For \(d=p-s\) one has
\begin{equation}\label{eq:D2closed}
 D^{(2)}(Y_{i,p},Y_{j,s})
       =\ct_{ij}(d-1)-\ct_{ij}(-d-1).
\end{equation}
In particular, when \(p>s\), this equals \(\ct_{ij}(p-s-1)\).
\end{lemma}
\begin{proof}
Abbreviate \(c(m)=\ct_{ij}(m)\). Substitution of \(a=-2h\) into
\eqref{eq:Na} gives
\begin{align*}
 -\sum_{h=1}^{M}N_{-2h}(i,p;j,s)
 ={}&\sum_{h=1}^{M}\bigl(c(d+1-2h)-c(d-1-2h)\bigr)\\
 &+\sum_{h=1}^{M}\bigl(c(-d-1-2h)-c(-d+1-2h)\bigr)\\
 ={}&c(d-1)-c(d-1-2M)\\
 &+c(-d-1-2M)-c(-d-1).
\end{align*}
Choose \(M\) large enough that both terminal arguments are nonpositive.
All subsequent summands vanish, and both terminal coefficients are zero.
This proves \eqref{eq:D2closed}. The argument includes \(d=0\) and
negative \(d\); the formula is visibly skew-symmetric.
\end{proof}
Define the integral skew form
\begin{equation}\label{eq:Omega}
 \Omega=2D^{(2)}-D_0
\end{equation}
on \(\PP\). We first identify its radical along the root-monomial
relations, then identify the resulting form on the quotient.

\subsection{A presentation of the spectral quotient}
Let \(\RR\subset\PP\) be the subgroup generated by the exponent vectors
\(a_{i,r}\) of the monomials \(A_{i,r}\). Here \(r-\xi_i\) is odd.
\begin{lemma}[The mesh quotient]\label{lem:mesh}
The assignment
\[
 e_{i,p}\longmapsto\beta_{i,p}
       =\tau^{(\xi_i-p)/2}\gamma_i
\]
induces an isomorphism of abelian groups
\begin{equation}\label{eq:meshquotient}
 \PP/\RR\ \xrightarrow{\ \sim\ }\ \LL.
\end{equation}
Consequently \(\deg_Q\) in \eqref{eq:degree-candidate} is surjective,
\(\ker\deg_Q=\RR\), and \(\deg_Q A_{i,r}=0\).
\end{lemma}
\begin{proof}
Write the coordinates of a root as \((x,y)\), with source coordinates
first, and let \(B\) be the source--sink incidence matrix. The products of
commuting reflections in the two colors act by
\[
 s_+(x,y)=(-x+By,y),\qquad s_-(x,y)=(x,-y+B^Tx).
\]
Thus \(\tau=s_+s_-\). The vectors on the initial slice are
\[
 \gamma_i=(e_i,0)\quad(i\in I_+),\qquad
 \gamma_j=(Be_j,e_j)\quad(j\in I_-).
\]
They are a \(\Z\)-basis of \(\LL\), since their change-of-basis matrix
is block triangular with identity diagonal blocks. Direct substitution gives
\begin{align}
 \tau\gamma_i+\gamma_i&=\sum_{j\sim i}\gamma_j
                          &&(i\in I_+),\label{eq:meshsource}\\
 \tau^{-1}\gamma_j+\gamma_j&=\sum_{i\sim j}\gamma_i
                          &&(j\in I_-).\label{eq:meshsink}
\end{align}
Their translates by powers of \(\tau\) are precisely
\[
 \beta_{i,r-1}+\beta_{i,r+1}
                  -\sum_{j\sim i}\beta_{j,r}=0.
\]
Hence \(\RR\) lies in the kernel of the proposed map.

In \(\PP/\RR\), the same recurrence expresses every generator in terms
of the slice \(e_{i,1}\) for \(i\in I_+\) and \(e_{j,0}\) for
\(j\in I_-\). For example, moving downward first determines the source
layer at \(-1\) from the layers at \(1,0\), then the sink layer at
\(-2\), and so on. Moving upward first determines the sink layer at
\(2\), then the source layer at \(3\). At each step the coefficient of
the new layer is \(1\); no division is involved. Thus the slice classes
generate the quotient. Their images are the basis \(\{\gamma_i\}\),
so they have no relation. This proves \eqref{eq:meshquotient}.
\end{proof}

The degree in this lemma is not the ordinary finite weight of a spectral
monomial. In the ordinary weight, \(\wt A_{i,r}=\alpha_i\); in the
\(Q\)-root degree, \(\deg_Q A_{i,r}=0\). The latter is the degree
which assigns a positive root to each fundamental class in \(\CQ\).

\begin{lemma}[Independence of root monomials]\label{lem:Aindependent}
The exponent vectors \(a_{i,r}\) are linearly independent for finite
integer combinations. In particular, a finite expression
\(m=m_+\prod A_{i,r}^{-v_{i,r}}\) has uniquely determined exponents
\(v_{i,r}\).
\end{lemma}
\begin{proof}
Suppose \(\sum c_{i,r}a_{i,r}=0\) is a nonzero finite relation and
choose the largest \(r\) with some \(c_{i,r}\ne0\). The coefficient of
\(e_{i,r+1}\) in the relation is \(c_{i,r}\): any other contributing
root monomial would have center \(r+1\) or \(r+2\), both larger than
the chosen maximum. Hence every \(c_{i,r}\) at the maximal level is
zero, a contradiction.
\end{proof}

\subsection{Descent of the commutation form}
\begin{lemma}[The two local boundary values]\label{lem:local-AY}
For each admissible pair one has
\begin{align}
 D_0(A_{i,r},Y_{j,s})
   &=2\delta_{ij}\bigl(\delta_{r-s,1}-\delta_{s-r,1}\bigr),
           \label{eq:A-D0}\\
 D^{(2)}(A_{i,r},Y_{j,s})
   &=\delta_{ij}\bigl(\delta_{r-s,1}-\delta_{s-r,1}\bigr).
           \label{eq:A-D2}
\end{align}
Thus \(\Omega(a_{i,r},e_{j,s})=0\).
\end{lemma}
\begin{proof}
Put \(c_{ij}(m)=\ct_{ij}(m)\). With the zero extension, the inverse
Cartan recurrence may be written at all integer arguments as
\begin{equation}\label{eq:delta-mesh}
 c_{ij}(m+1)+c_{ij}(m-1)-\sum_{k\sim i}c_{kj}(m)
                  =\delta_{ij}\delta_{m,0}.
\end{equation}
For \(m\ge1\) this is the recurrence, for \(m=0\) it is
\(c_{ij}(1)=\delta_{ij}\), and for \(m<0\) both sides are zero.
Expand the first argument \(A_{i,r}\) in \eqref{eq:D2closed}. The
four root-monomial contributions form \eqref{eq:delta-mesh} at
\(m=r-s-1\) and at \(m=s-r-1\), with opposite signs. The result is
\eqref{eq:A-D2}. Applying the same expansion to \(N_0\) in
\eqref{eq:Na} gives two copies of this boundary difference, proving
\eqref{eq:A-D0}. These formulas also agree with the specialized
\(A\)--\(Y\) relation of \cite[\S7]{FH}; the computation fixes the
ordered-pair signs without an implicit orientation convention.
\end{proof}

\begin{theorem}[Euler correction identity]\label{thm:kernel}
For all finite Laurent monomials \(m,n\),
\begin{equation}\label{eq:kernel-main}
 2D^{(2)}(m,n)-D_0(m,n)=\om(\deg_Qm,\deg_Qn).
\end{equation}
\end{theorem}
\begin{proof}
By \cref{lem:local-AY}, \(\Omega(a_{i,r},e_{j,s})=0\) for every
root-monomial generator and every spectral basis vector. Skew-symmetry
then gives \(\Omega(e_{j,s},a_{i,r})=0\). Hence \(\RR\) lies in the
radical in both variables, so \(\Omega\) descends to a well-defined skew
form on \(\PP/\RR\). By \cref{lem:mesh} this quotient is \(\LL\). It is
therefore enough to compute the descended form on the slice whose images
are \(\gamma_i\). If both vertices belong to the same color,
their slice positions agree, so \(D_0=D^{(2)}=0\). If \(i\in I_+\)
and \(j\in I_-\), then the slice positions are \(1,0\). Formula
\eqref{eq:D2closed} gives \(D^{(2)}=0\), while \eqref{eq:Na} gives
\[
 D_0(Y_{i,1},Y_{j,0})=-\ct_{ij}(2)=-B_{ij}.
\]
Thus \(\Omega(Y_{i,1},Y_{j,0})=B_{ij}\).

The Euler form \(\om\) has the same values on the \(\gamma\)-basis.
Indeed, \(\gamma_i=\alpha_i\) at a source and
\(\gamma_j=\alpha_j+\sum_{k\sim j}\alpha_k\) at a sink, so
\(\om(\gamma_i,\gamma_j)=B_{ij}\). For two sinks, the two sums over
common source neighbors occur with opposite signs and cancel. For two
sources the value is zero. Hence the two descended forms agree on a basis,
which proves the theorem.
\end{proof}
This proves the lattice and bilinear-form assertions of
\cref{thm:intro-kernel}. Notice that the argument uses the bipartite mesh,
not the absence of a trivalent vertex. It is uniform for finite ADE.

\begin{remark}[Dependency check]\label{rem:no-circularity}
The proof of \cref{thm:kernel} is entirely at the level of the spectral
lattice, the inverse quantum Cartan recurrence, and the bipartite
root-lattice quotient. Neither the graded-dimension theorem nor the
multiparameter realization proved later is used here. Thus the Euler
correction can be used subsequently to establish the grading and the
rank-two dimensions without circularity.
\end{remark}

\section{Character algebras as graded twists}\label{sec:twist}
\subsection{The moving parameter and the auxiliary extension}
Introduce \(w\) with \(w^2=u\), and write
\[
 R'=F[w^{\pm1}],\qquad K'=F(w),\qquad
 v=t_0u,\qquad v^{1/2}=t_0^{1/2}w.
\]
The extension \(R\subset R'\) is finite free of rank two. We denote by
\(*_v\) the ordinary one-parameter based-torus product
\[
 m*_v n=v^{D_0(m,n)/2}\,mn.
\]
For homogeneous degrees define a bicharacter
\begin{equation}\label{eq:kappa}
 \kappa(\alpha,\beta)=w^{\om(\alpha,\beta)}.
\end{equation}
For any \(\LL\)-graded algebra \(A\), its \(\kappa\)-twist has the
same graded module and multiplication
\(x\circ y=\kappa(\alpha,\beta)xy\) for homogeneous \(x,y\).
Bilinearity of \(\om\) gives
\[
 \kappa(\alpha,\beta)\kappa(\alpha+\beta,\gamma)
 =\kappa(\beta,\gamma)\kappa(\alpha,\beta+\gamma),
\]
so the new product is associative.

\begin{proposition}[Based-torus comparison]\label{prop:torus-twist}
Under the identity on commutative Laurent monomials, the toroidal torus
over \(R'\) is the \(\kappa\)-twist of the one-parameter torus at \(v\).
Explicitly, for homogeneous elements \(x,y\) of degrees \(\alpha,\beta\),
\begin{equation}\label{eq:twisted-product}
 x*y=w^{\om(\alpha,\beta)}(x*_v y).
\end{equation}
\end{proposition}
\begin{proof}
For monomials the toroidal scalar in \eqref{eq:based-product} is
\(t_0^{D_0/2}w^{2D^{(2)}}\), whereas the one-parameter scalar is
\(t_0^{D_0/2}w^{D_0}\). Their quotient is
\(w^{2D^{(2)}-D_0}=w^{\om(\alpha,\beta)}\) by
\cref{thm:kernel}. Extend by bilinearity. Homogeneity is needed so that the
same scalar multiplies every monomial pair in the two elements.
\end{proof}
The fourth-root extension of the original parameter \(t\) is not part of
the statement of the final isomorphism. It is a device for comparing the
symmetric based products; \cref{sec:global} works over \(K\) itself.

\subsection{Screening generators and fundamental characters}
\begin{proposition}[Character comparison]\label{prop:characters}
After the parameter substitution and scalar extension to \(K'\),
\begin{equation}\label{eq:characters}
 F_\infty(m_+)\big|_{\eqref{eq:specialization}}=F_v(m_+)
\end{equation}
as commutative-monomial expansions. Both sides are homogeneous of degree
\(\deg_Qm_+\). In particular,
\begin{equation}\label{eq:trunc-character}
 Z_{i,p}=\operatorname{tr}_Q F_v(Y_{i,p})
 \qquad((i,p)\in\IhQ)
\end{equation}
after the same scalar extension.
\end{proposition}
\begin{proof}
Every term of a character has the same \(Q\)-degree by
\cref{lem:lower,lem:mesh}.  The two summands of
\[
B_{i,p}=Y_{i,p}+Y_{i,p}A_{i,p+1}^{-1}
\]
have the same degree, and the remaining screening generators are
homogeneous Laurent variables.

Fix \(i\).  Let \(\mathscr K^{\mathrm{tor}}_i\) be the specialized
toroidal screening subalgebra and \(\mathscr K^v_i\) the corresponding
one-parameter screening subalgebra.  Both are generated by the same
homogeneous commutative-monomial expressions.  For any ordered word
\(G_1,\ldots,G_\ell\) in these homogeneous generators,
\cref{prop:torus-twist} gives
\[
G_1*\cdots*G_\ell
=
w^{c(G_1,\ldots,G_\ell)}
\bigl(G_1*_v\cdots*_vG_\ell\bigr)
\]
for an integer \(c(G_1,\ldots,G_\ell)\).  The factor on the right is a
unit of \(K'\).  Hence the identity on commutative Laurent expansions
identifies the two screening subalgebras as graded \(K'\)-vector spaces:
\[
\mathscr K^{\mathrm{tor}}_i=\mathscr K^v_i
\quad\text{inside the common Laurent-monomial vector space.}
\tag{4.12}
\]
This assertion is made separately for each \(i\); no base-change
interchange with an infinite intersection is being used.

Now let
\[
\overline F_\infty(m_+)
=
F_\infty(m_+)\big|_{\eqref{eq:specialization}}.
\]
Since \(F_\infty(m_+)\) belongs to every toroidal screening subalgebra,
(4.12) shows directly that
\[
\overline F_\infty(m_+)\in\bigcap_i\mathscr K^v_i.
\]
Its coefficient of \(m_+\) remains \(1\).  Before specialization
\(m_+\) is the only dominant exponent vector in \(F_\infty(m_+)\);
specialization can cancel existing terms but cannot create a new exponent
vector.  Thus \(m_+\) remains the unique dominant monomial of
\(\overline F_\infty(m_+)\).

The one-parameter screening intersection has the basis
\(\{F_v(m)\}\), characterized triangularly by dominant monomials
\cite{H,FH}.  This characterization survives extension of scalars to
\(K'\): if an element of the intersection has exactly one dominant
monomial \(m_+\) with coefficient \(1\), subtracting \(F_v(m_+)\) leaves
an element with no dominant monomial, hence zero.  Therefore
\[
\overline F_\infty(m_+)=F_v(m_+).
\]
Applying the same linear truncation to both sides gives
\eqref{eq:trunc-character}.
\end{proof}

\begin{remark}\label{rem:basechange-intersection}
The proof deliberately avoids asserting a formal identity
\[
(\bigcap_i\mathscr K_{i,\infty})\otimes K'
=
\bigcap_i(\mathscr K_{i,\infty}\otimes K'),
\]
which would require a separate flatness argument for intersections.
Instead, it follows one already constructed element through each screening
subalgebra and then invokes the one-parameter unique-dominant
characterization.  It also uses only exponent support and screening
membership; toroidal coefficients are never interpreted as ordinary
composition multiplicities.
\end{remark}

\begin{theorem}[The graded character algebra]\label{thm:graded-twist}
After extension to \(K'\), the algebra \(\GG_R(Q)\) is the graded
\(\kappa\)-twist of the truncated one-parameter algebra
\(\mathcal K_v(\CQ)\). Its positive root grading is induced by
\(\deg_Q\), and
\begin{equation}\label{eq:graded-dim}
 \dim_K\bigl(\GG_R(Q)\otimes_RK\bigr)_\gamma=p(\gamma)
 \qquad(\gamma\in\LL_+),
\end{equation}
where \(p(\gamma)\) is the Kostant partition number.
\end{theorem}
\begin{proof}
The two algebras are generated by the same truncated fundamental
expressions by \cref{prop:characters}. These generators are homogeneous
and their products differ by the factors in \eqref{eq:twisted-product}.
Thus their underlying graded vector spaces over \(K'\) agree. All
generator degrees are positive roots, because \((i,p)\in\IhQ\).
Hernandez--Leclerc identify the one-parameter algebra, over
\(\Q(v^{1/2})\), with the quantized positive coordinate algebra and
hence with the positive enveloping algebra after extension of scalars
\cite[Theorem 6.1]{HL}. Their identification sends the fundamental
classes to rescaled dual PBW generators of their root degree. It follows
that the dimension of a homogeneous component is \(p(\gamma)\).
The embedding in the based torus respects degree already over \(R\),
and extension from \(K\) to \(K'\) does not change vector-space
dimensions. This gives \eqref{eq:graded-dim} over \(K\).
\end{proof}

\subsection{A local standard basis and the special fiber}
For specialization arguments use
\begin{equation}\label{eq:localring}
 \OO=F[u^{\pm1}]_{(u-1)},\qquad \OO'=\OO\otimes_RR'.
\end{equation}
The ring \(\OO\) is a discrete valuation ring. The extension
\(\OO\to\OO'\) is finite free and faithfully flat. The latter is
semilocal, with the two possible residues \(w=1,-1\); we do not replace
it by a single branch in a descent argument. The ideal used in \(\OO\)
is \((u-1)\), not \((t-1)\) in a ring already containing \(u\).

Fix a convex order on \(\Delta_+\), equivalently an order of the
fundamental positions in \(\IhQ\) used for the standard basis in
\cite[\S\S5--6]{HL}. For \(\mathbf a\in\N^{\Delta_+}\), let
\(\mathsf S_{\mathbf a}\) be the corresponding ordered product of the
\(Z_{i,p}\), normalized by its invertible highest coefficient.

\begin{proposition}[Local standard basis]\label{prop:localbasis}
The elements \(\mathsf S_{\mathbf a}\) form an \(\OO\)-basis of
\(\GG_{\OO}(Q):=\GG_R(Q)\otimes_R\OO\). In particular, this algebra
is \(\OO\)-free and its degree-\(\gamma\) component has rank
\(p(\gamma)\). Its special fiber is the one-parameter truncated
Grothendieck algebra at \(t_0\):
\begin{equation}\label{eq:specialfiberG}
 \GG_{\OO}(Q)/(u-1)\GG_{\OO}(Q)
       \simeq\mathcal K_{t_0}(\CQ)\otimes_{\Q(t_0^{1/2})}F.
\end{equation}
\end{proposition}
\begin{proof}
The ordinary one-parameter standard basis consists of ordered
fundamental products, normalized by quantum-torus commutation factors; see
\cite[\S5 and Theorem 6.1]{HL}. After the substitution
\(v^{1/2}=t_0^{1/2}w\), these normalization factors are monomials in
\(t_0^{1/2}\) and \(w\), hence units of \(\OO'\). No assertion about an
arbitrary rational function of \(v^{1/2}\) is needed.

The character comparison \eqref{eq:characters}, initially proved over
\(K'\), identifies the same commutative-monomial expansions on the two
sides. Since the specialized toroidal generators lie in the based
\(\OO'\)-torus, the corresponding one-parameter expressions do as well.
By \cref{prop:torus-twist}, every ordered homogeneous product differs from
its one-parameter counterpart by a power of \(w\), again an
\(\OO'\)-unit. Consequently the elements \(\mathsf S_{\mathbf a}\)
become, after extension to \(\OO'\), unit multiples of the ordinary
standard-basis elements and hence form an \(\OO'\)-basis.

Consider the map from the free \(\OO\)-module on these symbols to
\(\GG_{\OO}(Q)\). Its kernel and cokernel become zero after the
faithfully flat extension \(\OO'\). They are consequently zero
already over \(\OO\). This proves the asserted basis and ranks,
without relying on a triangular expansion in the deformed
\(F_\infty\)-basis.

Under \(u=1\), the specialized fundamental characters are the
one-parameter ones at \(t_0\). This follows either from the unique
dominant construction or by choosing \(w=1\) in the preceding character
comparison. The ordered standard products specialize to their ordinary
standard basis. The resulting surjection to the one-parameter algebra is
therefore an isomorphism of the based modules, and of algebras. This proves
\eqref{eq:specialfiberG}.
\end{proof}

\begin{corollary}[Small homogeneous dimensions]\label{cor:smalldim}
If \(r\sim s\), the components of degrees
\(\alpha_r+\alpha_s\), \(2\alpha_r+\alpha_s\), and
\(\alpha_r+2\alpha_s\) have dimension two over \(K\). If
\(i\ne j\) and \(i\not\sim j\), the component of degree
\(\alpha_i+\alpha_j\) has dimension one.
\end{corollary}
\begin{proof}
Only roots supported on the indicated vertices can occur in a Kostant
partition of these degrees. For an adjacent pair, the available positive
roots are \(\alpha_r,\alpha_s,\alpha_r+\alpha_s\), giving exactly
two partitions of each listed degree. For a nonadjacent pair only the two
simple roots are available, giving one partition. Apply
\eqref{eq:graded-dim}.
\end{proof}

\subsection{Source classes and highest pairings}
\begin{proposition}[Simple-root normalizations]\label{prop:source-base}
If \(r\in I_+\), then \(Z_r=Y_{r,1}\). If \(r\to s\) and
\(m_s=Y_{\kappa_s,p_s}\) is the highest monomial of \(Z_s\), then
\begin{equation}\label{eq:base-pair}
 (D_0,D^{(2)})(Y_{r,1},m_s)=(1,1).
\end{equation}
For distinct nonadjacent simple roots with highest monomials \(m_i,m_j\),
one has
\begin{equation}\label{eq:nonadjbase}
 D_0(m_i,m_j)=D^{(2)}(m_i,m_j)=0.
\end{equation}
\end{proposition}
\begin{proof}
At a source, \(\gamma_r=\alpha_r\), so the simple-root position is
\((r,1)\). The standard initial Kirillov--Reshetikhin class with highest
monomial \(Y_{i,p}Y_{i,p+2}\cdots Y_{i,\xi_i}\) has this single
monomial as its \(Q\)-truncated character; this is the initial-character
identity in the proof of \cite[Theorem 6.1]{HL}. With \(p=\xi_r=1\)
it gives \(Z_r=Y_{r,1}\) in the one-parameter algebra, and
\cref{prop:characters} gives the same equality toroidally. No assumption
of thinness of the full fundamental module is made.

Every point of \(\IhQ\) is weakly below its top height. A simple-root
point for a sink cannot have spectral coordinate \(1\), since the only
level-zero points there are the source tops and their roots are source
simple roots. Hence \(p_s<1\). The inverse-Cartan/root-pairing formula
of \cite[Proposition 3.2]{HL} gives
\[
 D_0(m_s,Y_{r,1})=(\alpha_s,\alpha_r)=-1.
\]
Skew-symmetry proves \(D_0(Y_{r,1},m_s)=1\). Since
\(\deg_QY_{r,1}=\alpha_r\), \(\deg_Qm_s=\alpha_s\), and
\(\om(\alpha_r,\alpha_s)=1\), \cref{thm:kernel} gives
\(2D^{(2)}=D_0+1=2\), proving \eqref{eq:base-pair}.

For nonadjacent simple roots, the same Hernandez--Leclerc formula gives
\(D_0(m_i,m_j)=0\) when their spectral positions differ. If the
positions agree, the kernel \eqref{eq:Na} is zero directly; no formula
requiring a strict spectral inequality is used in this case. Finally
\(\om(\alpha_i,\alpha_j)=0\), so \cref{thm:kernel} gives the
second equality in \eqref{eq:nonadjbase}.
\end{proof}

\section{Local spectral rigidity}\label{sec:local}
In this section products of elements of \(\GG_K(Q)\) are always toroidal
products. Powers and conjugation are taken with this product. A monomial
inverse is used only in the ambient torus.

\subsection{An annihilator and its spectral bound}
The elementary principle behind the argument is most naturally stated for
an arbitrary based torus.
\begin{lemma}[Krylov dimension and monomial eigenvalues]\label{lem:krylov}
Let \(\TT_k\) be a based quantum torus over a field \(k\), let \(X\)
be a nonzero monomial, and let \(Z\in\TT_k\) be a finite Laurent sum.
Suppose the distinct eigenvalues of \(\Ad_X\) on the nonzero monomial
blocks of \(Z\) are \(\lambda_1,\ldots,\lambda_d\). Then
\[
 \dim_k\operatorname{span}_k\{Z,\Ad_XZ,\Ad_X^2Z,\ldots\}=d.
\]
If the words \(X^2Z,XZX,ZX^2\) lie in a vector space of dimension at
most two, then \(d\le2\).
\end{lemma}
\begin{proof}
Write \(Z=\sum_{j=1}^d Z_j\), grouping its monomials by eigenvalue.
The \(Z_j\) have disjoint supports and are linearly independent. The
coefficients of \(\Ad_X^aZ\), for \(0\le a\le d-1\), on this
basis are the rows of the Vandermonde matrix
\((\lambda_j^a)\), whose determinant is nonzero. Higher iterates remain
in the span of the \(Z_j\), giving the first assertion.

For the second assertion, choose a nontrivial dependence
\[
 A X^2Z+B XZX+C ZX^2=0.
\]
Right multiplication in the torus by \(X^{-2}\) gives
\begin{equation}\label{eq:adjointpoly}
 (A\Ad_X^2+B\Ad_X+C)Z=0.
\end{equation}
The polynomial \(AT^2+BT+C\) is nonzero. On each nonzero monomial
block it vanishes at that block's eigenvalue, so it has at least \(d\)
distinct roots. Thus \(d\le2\). No inverse of \(X\) is asserted to
belong to the original subalgebra containing the three words.
\end{proof}

\subsection{The wall number}
Fix a source--sink edge \(r\to s\) and put
\[
 X=Z_r=Y_{r,1},\qquad Z=Z_s,\qquad m_+=Y_{\kappa_s,p_s}.
\]
For a monomial in \(Z\), use its unique lower-support expression to set
\begin{equation}\label{eq:wallnumber}
 m=m_+\prod A_{i,p}^{-v_{i,p}(m)},\qquad
 \delta_r(m)=v_{r,0}(m)-v_{r,2}(m).
\end{equation}
Uniqueness follows from \cref{lem:Aindependent}. The two indices \(0,2\)
are the spectral neighbors of the source height \(1\).

\begin{lemma}[Net wall formula]\label{lem:netwall}
Every monomial occurring in \(Z\) satisfies
\begin{align}
 (D_0,D^{(2)})(X,m)
       &=(1-2\delta_r(m),\ 1-\delta_r(m)),\label{eq:netwall}\\
 \Ad_X(m)&=\lambda_{\delta_r(m)}m,
 \qquad \lambda_\delta=t_0^{1-2\delta}u^{2-2\delta}.
       \label{eq:lambda-delta}
\end{align}
\end{lemma}
\begin{proof}
The highest pairing is \((1,1)\) by \eqref{eq:base-pair}. By
\cref{lem:local-AY}, every root factor other than \(A_{r,0}\) and
\(A_{r,2}\) pairs trivially with \(X\). The form on
\((X,A_{r,0})\) is \((2,1)\), and on \((X,A_{r,2})\) it is
\(-(2,1)\). Since the root factors in \eqref{eq:wallnumber} occur
inversely, their total contribution is \(-\delta_r(m)(2,1)\).
This proves \eqref{eq:netwall}, and the based-torus commutation rule gives
\eqref{eq:lambda-delta}.
\end{proof}
The number \(\delta_r(m)\) is an integer; no separate bound on the two
multiplicities in \eqref{eq:wallnumber} is assumed. In particular, neither
thinness nor a tableau formula enters the calculation.

\begin{lemma}[Integral spectral-block decomposition]\label{lem:integralblocks}
Let \(Z\in\GG_{\OO}(Q)\) be a finite Laurent sum and let \(X\) be a
monomial.  Assume every monomial \(m\) in \(Z\) is an
\(\Ad_X\)-eigenvector and that its eigenvalue is indexed by an integer
\(\delta(m)\).  Then
\[
 Z=\sum_{\delta}Z[\delta],\qquad
 Z[\delta]=\sum_{\delta(m)=\delta}c_m\,m,
\]
is a direct decomposition in the ambient based \(\OO\)-torus.  Each
\(Z[\delta]\) is defined over \(\OO\), and specialization at \(u=1\) is
obtained coefficientwise:
\[
 Z[\delta]\longmapsto
 \sum_{\delta(m)=\delta}c_m(1)\,m.
\]
In particular, no denominator involving differences of eigenvalues is
introduced.
\end{lemma}
\begin{proof}
The ambient based torus is a free \(\OO\)-module on its Laurent monomial
basis.  Partitioning this basis by the integer \(\delta(m)\) therefore
gives a direct sum decomposition of the underlying \(\OO\)-module.
The block \(Z[\delta]\) is simply the projection onto the corresponding
basis summand; it is not obtained by a rational polynomial in
\(\Ad_X\).  Hence all coefficients remain in \(\OO\), and reduction
modulo \(u-1\) is coefficientwise.
\end{proof}

\begin{corollary}[Criterion for block survival]\label{cor:blocksurvival}
In the setting of Lemma~\ref{lem:integralblocks}, \(Z[\delta]\) vanishes
in the \(u=1\) fiber if and only if every coefficient of that block is
divisible by \(u-1\).  If all other blocks specialize to a single
\(\Ad_X\)-eigenvalue, then vanishing of \(Z[\delta]\) forces the whole
specialized element \(Z|_{u=1}\) to lie in that one eigenspace.
\end{corollary}
\begin{proof}
The first assertion follows from freeness of the Laurent monomial basis
over the DVR \(\OO\).  The second is immediate from the direct
coefficientwise decomposition.
\end{proof}

\begin{theorem}[Two-point local spectrum]\label{thm:two}
For every source--sink edge in a finite simply-laced bipartite quiver,
\begin{equation}\label{eq:two}
 \Spec_X(Z)=\{t_0t,t_0^{-1}\}.
\end{equation}
Equivalently, \(\delta_r(m)\in\{0,1\}\) for every occurring monomial,
and the exponent pairs are exactly \((1,1)\) and \((-1,0)\).
\end{theorem}

\begin{remark}[Order of scalar extension and specialization]
The spectral bound is proved over the generic field \(K\), whereas block
survival is checked over the integral local form \(\OO\).  This causes no
conflict: the blocks are first defined by Laurent-monomial support over
\(\OO\), then extended to \(K\) for the generic eigenvalue argument, and
finally reduced modulo \(u-1\).  Thus no eigenblock is created by scalar
extension and no projector denominator needs to be specialized.
\end{remark}

\begin{proof}
The three words \(X^2Z,XZX,ZX^2\) have degree
\(2\alpha_r+\alpha_s\) in \(\GG_K(Q)\). Its dimension is two by
\cref{cor:smalldim}. By \cref{lem:krylov}, at most two distinct
\(\Ad_X\)-eigenvalues occur. The eigenvalues
\(\lambda_\delta\) in \eqref{eq:lambda-delta} are pairwise distinct,
since \(t_0,u\) are algebraically independent. Thus at most two integer
wall numbers occur.

The highest monomial has coefficient one and wall number zero. Hence the
block \(Z[0]\) is nonzero and \(\lambda_0=t_0t\) occurs. If this
were the only block over \(K\), then \(XZ=t_0tZX\). This identity is
already an identity of Laurent sums over \(\OO\), so it specializes at
\(u=1\). In the one-parameter fiber, the simple-root classes correspond
to nonzero scalar multiples of adjacent Chevalley generators
\cite[\S7.1]{HL}. Those generators admit no relation of degree
\(\alpha_r+\alpha_s\); their two ordered words are independent.
The specialized equality \(XZ=t_0ZX\) is impossible. Therefore there
is a second generic block, say
\[
 Z=Z[0]+Z[d],\qquad d\ne0.
\]

By Lemma~\ref{lem:integralblocks}, these blocks are defined over the
ambient based \(\OO\)-torus by direct monomial grouping; no rational
spectral projector is involved.  We do not assert that either block
belongs to \(\GG_{\OO}(Q)\).  If the second block vanished in the
special fiber, Corollary~\ref{cor:blocksurvival} would leave only the
\(\delta=0\) eigenspace and hence force the same forbidden equality
\(XZ=t_0ZX\).  Thus its specialization is nonzero.  The first block also
survives because it contains the highest monomial with coefficient one.

In the one-parameter torus, the two specialized blocks have eigenvalues
\(t_0\) and \(t_0^{1-2d}\). The ordinary adjacent Serre relation gives
\[
 (\Ad_X-t_0)(\Ad_X-t_0^{-1})Z=0.
\]
Its two terms can be separated by their disjoint monomial supports.
Since the second block is nonzero and \(d\ne0\), one must have
\(t_0^{1-2d}=t_0^{-1}\). Hence \(d=1\). Formula
\eqref{eq:netwall} gives the two asserted pairs, both of which occur.
\end{proof}

\begin{remark}
The proof allows coefficients of individual monomials to vanish at
\(u=1\). It requires only that the second entire block cannot vanish.
The bound on the number of generic blocks has already been proved before
specialization. This is the precise replacement for a termwise
support-preservation hypothesis in a non-thin character.
\end{remark}

\subsection{The two oriented Serre relations}
Put \(a=t_0t\), \(b=t_0^{-1}\), and write
\(Z=P+Q_1\) for the two nonzero monomial blocks. The symbol \(Q_1\)
is used here to distinguish the second block from the quiver \(Q\).
Then \(XP=aPX\), \(XQ_1=bQ_1X\).

\begin{proposition}[Forward Serre relation]\label{prop:forward}
One has
\begin{equation}\label{eq:forward}
 X^2Z-(t_0t+t_0^{-1})XZX+tZX^2=0.
\end{equation}
\end{proposition}
\begin{proof}
For a monomial \(m\) with eigenvalue \(\lambda\in\{a,b\}\),
\[
 X^2m-(a+b)XmX+ab\,mX^2
       =(\lambda^2-(a+b)\lambda+ab)mX^2=0.
\]
Sum over the monomial expansion of \(Z\), and use \(ab=t\).
\end{proof}

\begin{lemma}[Separation of the pure square blocks]\label{lem:puresquares}
Let \(Z_s=P+Q\) be the two nonzero \(\Ad_{Z_r}\)-eigenblocks, with
\[
\Ad_{Z_r}(P)=aP,\qquad \Ad_{Z_r}(Q)=bQ,\qquad a\neq b.
\]
Then \(P^2,Q^2\neq0\) in the ambient quantum torus and are eigenvectors
with eigenvalues \(a^2,b^2\), respectively.  For the generic parameters
of this paper these eigenvalues are distinct, so the two pure-square
components of any ambient-torus relation vanish separately.
\end{lemma}
\begin{proof}
The ambient based quantum torus is a domain, hence \(P^2,Q^2\neq0\).
Since \(\Ad_{Z_r}\) is an algebra automorphism,
\[
\Ad_{Z_r}(P^2)=a^2P^2,\qquad \Ad_{Z_r}(Q^2)=b^2Q^2.
\]
Here \(a=t_0t\) and \(b=t_0^{-1}\); algebraic independence of \(t_0,t\)
gives \(a^2\neq b^2\).
\end{proof}

\begin{lemma}[Reverse-Serre rigidity]\label{lem:reverse}
Let \(X\) be an invertible monomial in a based torus over a field, and
let \(Z=P+Q_1\) with nonzero blocks satisfying
\(XP=aPX\), \(XQ_1=bQ_1X\). Assume \(a^2,ab,b^2\) are pairwise
distinct and \(Z^2X,ZXZ,XZ^2\) are linearly dependent. If the ambient
torus is a domain, their unique normalized relation is
\begin{equation}\label{eq:reverse-general}
 Z^2X-(a^{-1}+b^{-1})ZXZ+(ab)^{-1}XZ^2=0.
\end{equation}
\end{lemma}
\begin{proof}
Take an actual nonzero relation
\begin{equation}\label{eq:actual-reverse}
 A Z^2X+B ZXZ+C XZ^2=0.
\end{equation}
The pure \(P^2\), mixed \(PQ_1,Q_1P\), and pure \(Q_1^2\)
contributions have respective \(\Ad_X\)-eigenvalues
\(a^2,ab,b^2\). Multiplication by \(X\) does not change them. Thus the
three groups separate; in particular, neither mixed group can contaminate
a pure group. With \(X\) moved to the right, the pure contributions are
\[
 \begin{array}{c|ccc}
 &Z^2X&ZXZ&XZ^2\\ \hline
 P^2&P^2X&aP^2X&a^2P^2X\\
 Q_1^2&Q_1^2X&bQ_1^2X&b^2Q_1^2X.
 \end{array}
\]
Because the torus is a domain these pure factors are nonzero. Hence
\begin{equation}\label{eq:reverse-coeff}
 A+Ba+Ca^2=0,\qquad A+Bb+Cb^2=0.
\end{equation}
If \(A=0\), division by \(a,b\) and subtraction give \(C=0\), then
\(B=0\), a contradiction. Normalize \(A=1\). The equations imply
\[
 B=-C(a+b),\qquad 1-Cab=0,
\]
so \(C=(ab)^{-1}\) and \(B=-(a^{-1}+b^{-1})\).
The coefficients of the actual relation \eqref{eq:actual-reverse} are
therefore uniquely fixed, proving \eqref{eq:reverse-general}. No
commutation relation between \(P\) and \(Q_1\) was assumed: the mixed
part satisfies the result because it was part of an existing relation.
\end{proof}

\begin{corollary}[Reverse Serre relation]\label{cor:reverse}
For a source--sink edge \(r\to s\),
\begin{equation}\label{eq:reverse}
 Z_s^2Z_r-(t_0+t_0^{-1}t^{-1})Z_sZ_rZ_s+t^{-1}Z_rZ_s^2=0.
\end{equation}
\end{corollary}
\begin{proof}
The three words have degree \(\alpha_r+2\alpha_s\) and are linearly
dependent by \cref{cor:smalldim}. Both blocks are nonzero by
\cref{thm:two}. Moreover,
\[
 a^2=t_0^2t^2,\qquad ab=t,\qquad b^2=t_0^{-2}
\]
are pairwise distinct over \(K\). Apply \cref{lem:reverse}.
\end{proof}

\section{The multiparameter presentation and realization}\label{sec:global}
\subsection{Highest terms and nonadjacent generators}
\begin{lemma}[Highest-term ratio]\label{lem:highratio}
Suppose \(Z_iZ_j=cZ_jZ_i\) for a scalar \(c\) and let \(m_i,m_j\)
be the highest monomials. Then
\[
 c=t_0^{D_0(m_i,m_j)}t^{D^{(2)}(m_i,m_j)}.
\]
For nonadjacent distinct simple roots, \(Z_iZ_j=Z_jZ_i\).
\end{lemma}
\begin{proof}
The highest monomial of each factor has coefficient one. A term obtained
by lowering either factor cannot produce \(m_im_j\): the ratio from
that highest product would be a product of inverse root monomials with
nonnegative exponents. It is nontrivial by \cref{lem:Aindependent},
or simply by its strictly negative ordinary-weight height. Thus the
coefficient of \(m_im_j\) comes only from the two highest terms. Their
coefficients in the two product orders are inverse based-torus factors.
Their ratio is the stated scalar.

For nonadjacent vertices the two products are nonzero and lie in the
one-dimensional component of degree \(\alpha_i+\alpha_j\), so such
an equality exists. By \eqref{eq:nonadjbase}, its scalar is \(1\).
\end{proof}
Thus the nonadjacent parameters are trivial for the present specialization.
General Cartan-type matrices may allow nontrivial reciprocal parameters on
nonedges, but they do not occur in this bipartite realization.

\subsection{Reconstruction of the Cartan parameters}
Let \(C=(c_{ij})\) be the finite Cartan matrix. In the
Pei--Hu--Rosso convention, a Cartan-type parameter matrix satisfies
\(q_{ij}q_{ji}=q_{ii}^{c_{ij}}\) for \(i\ne j\). Its positive part
has relations
\begin{align}
 E_iE_j-q_{ij}E_jE_i&=0&& (c_{ij}=0),\label{eq:HPRzero}\\
 E_i^2E_j-(1+q_{ii})q_{ij}E_iE_jE_i
                 +q_{ii}q_{ij}^2E_jE_i^2&=0&& (c_{ij}=-1).
                 \label{eq:HPRserre}
\end{align}
See \cite{HPR}. Put \(v=t_0u\). The parameters in
\eqref{eq:intro-q} may equivalently be written
\begin{equation}\label{eq:qcompact}
 q_{ii}=v^2,\qquad q_{ij}=v^{c_{ij}}u^{\om(\alpha_i,\alpha_j)}
                      \quad(i\ne j).
\end{equation}
For an edge the ratio of the two local eigenvalues is
\((t_0t)/(t_0^{-1})=v^2\), and the second eigenvalue is
\(q_{rs}=t_0^{-1}\). Reversing the edge gives
\(q_{sr}=v^{-2}q_{rs}^{-1}\), as in \eqref{eq:intro-q}.

\begin{proposition}[Ordered Cartan-parameter check]\label{prop:hprcompat}
For an oriented adjacent edge \(i\to j\), put
\[
q_{ii}=q_{jj}=t_0^2t,\qquad
q_{ij}=t_0^{-1},\qquad q_{ji}=t_0^{-1}t^{-1}.
\]
Then
\[
q_{ij}q_{ji}=q_{ii}^{-1}=q_{ii}^{a_{ij}},
\]
and the standard simply-laced Cartan-type Serre coefficients are
\[
(1+q_{ii})q_{ij}=t_0^{-1}+t_0t,\qquad
q_{ii}q_{ij}^2=t,
\]
and
\[
(1+q_{jj})q_{ji}=t_0+t_0^{-1}t^{-1},\qquad
q_{jj}q_{ji}^2=t^{-1}.
\]
Thus the forward and reverse toroidal Serre relations use the same
ordered matrix \(\mathbf q\), with no interchange of \(q_{ij}\) and
\(q_{ji}\).
\end{proposition}
\begin{proof}
All four identities follow by direct substitution.  The first is
\[
t_0^{-1}(t_0^{-1}t^{-1})
=t_0^{-2}t^{-1}
=(t_0^2t)^{-1},
\]
and the remaining identities are immediate expansions.
\end{proof}

\begin{remark}[Nonadjacent entries]
If \(a_{ij}=0\), the Euler-correction identity gives zero toroidal
exponent on nonadjacent simple roots.  Hence
\[
q_{ij}=q_{ji}=1,
\]
so the nonadjacent Cartan relation is ordinary commutation.
\end{remark}

\begin{proposition}[Cartan relations]\label{prop:Cartan-rel}
The parameters \eqref{eq:qcompact} satisfy the Cartan constraint, and
the assignment \(E_i\mapsto Z_i\) satisfies all relations
\eqref{eq:HPRzero}--\eqref{eq:HPRserre}.
\end{proposition}
\begin{proof}
Skewness of \(\om\) gives
\(q_{ij}q_{ji}=v^{2c_{ij}}=q_{ii}^{c_{ij}}\).
On an edge \(r\to s\),
\begin{align*}
 (1+q_{rr})q_{rs}&=t_0^{-1}+t_0t,&
 q_{rr}q_{rs}^2&=t,\\
 (1+q_{ss})q_{sr}&=t_0+t_0^{-1}t^{-1},&
 q_{ss}q_{sr}^2&=t^{-1}.
\end{align*}
The claims follow from \cref{prop:forward,cor:reverse,lem:highratio}.
\end{proof}

\subsection{A local cocycle and a PBW form}
Let \(\bar q_{ii}=v^2\) and \(\bar q_{ij}=v^{c_{ij}}\) for
\(i\ne j\). The ratio
\[
 \chi_{ij}=q_{ij}/\bar q_{ij}=u^{\om(\alpha_i,\alpha_j)}
\]
is an \(\OO\)-unit and is \(1\) at \(u=1\). Fix a total order
on the simple roots and define
\begin{equation}\label{eq:sigma}
 \sigma\left(\sum a_i\alpha_i,\sum b_j\alpha_j\right)
       =\prod_{i<j}\chi_{ij}^{a_ib_j}.
\end{equation}
This is an \(\OO^\times\)-valued bicharacter, with
\[
 \frac{\sigma(\alpha_i,\alpha_j)}{\sigma(\alpha_j,\alpha_i)}
       =\chi_{ij},\qquad \sigma\big|_{u=1}=1.
\]
Unlike the symmetric comparison \eqref{eq:kappa}, this cocycle is defined
over \(\OO\) without adjoining \(w\).

\begin{lemma}[Cocycle convention]\label{lem:cocycle}
Twisting the homogeneous product of the one-parameter positive part by
\(x\star y=\sigma(|x|,|y|)xy\) changes its parameters to
\(q_{ij}=\bar q_{ij}\sigma_{ij}/\sigma_{ji}\), where
\(\sigma_{ij}=\sigma(\alpha_i,\alpha_j)\).
\end{lemma}
\begin{proof}
For an adjacent pair, bicharacter multiplicativity gives
\begin{align*}
 E_i\star E_i\star E_j&=\sigma_{ii}\sigma_{ij}^2E_i^2E_j,\\
 E_i\star E_j\star E_i&=\sigma_{ii}\sigma_{ij}\sigma_{ji}E_iE_jE_i,\\
 E_j\star E_i\star E_i&=\sigma_{ii}\sigma_{ji}^2E_jE_i^2.
\end{align*}
Multiply the ordinary Serre relation by \(\sigma_{ii}\sigma_{ij}^2\)
and rewrite the products. Its middle coefficient becomes
\((1+\bar q_{ii})\bar q_{ij}\sigma_{ij}/\sigma_{ji}\), and its
last coefficient becomes
\(\bar q_{ii}\bar q_{ij}^2(\sigma_{ij}/\sigma_{ji})^2\).
This is \eqref{eq:HPRserre} with the stated parameter. The nonadjacent
two-term relation transforms by the same ratio. Applying the inverse twist
recovers the original relations; hence this calculation identifies the
presentations, not merely a quotient of one presentation.
\end{proof}

\begin{proposition}[PBW form at the special fiber]\label{prop:PBW}
There is an \(\OO\)-algebra \(U_{\mathbf q,\OO}^+\) with a free
ordered PBW \(\OO\)-basis, generic fiber \(U_{\mathbf q}^+\), and
special fiber \(U_{t_0}^+(\mathfrak g)\).
\end{proposition}
\begin{proof}
Take a standard ordered PBW basis of the one-parameter positive part
with parameter \(v=t_0u\).  We use the local lattice spanned by these
ordered PBW vectors over \(\OO\).  The structure constants needed for
multiplication lie in the localization because their denominators are the
standard quantum normalizing factors evaluated at the generic parameter
\(v\); at \(u=1\) these specialize to nonzero rational functions of the
transcendental variable \(t_0\).  Hence the required normalizing factors
are \(\OO\)-units.  We do not claim that an arbitrary nonzero rational
function of \(v\) is a unit at the special fiber.

The ordinary PBW theorem gives freeness of this local lattice.  The
multiparameter cocycle/shuffle framework is used only to identify the
twisted generic multiplication with the Cartan-type presentation
\cite{HPR}.

Twist the multiplication by \eqref{eq:sigma}. All twist values are
\(\OO\)-units, so the underlying free module is unchanged. By
\cref{lem:cocycle} and its inverse-twist argument, the generic fiber is
the Cartan-type positive part with parameters \(\mathbf q\). At
\(u=1\), the cocycle is identically one and \(v=t_0\), proving the
special-fiber assertion. This constructs the local form; it is not an
assertion that a particular integral form over all of \(R\) was already
proved in \cite{HPR}.
\end{proof}

\subsection{Specialization and the isomorphism}
\begin{lemma}[Degree-one special fiber]\label{lem:simplefiber}
Under the Hernandez--Leclerc identification of the fiber
\eqref{eq:specialfiberG} with \(U_{t_0}^+\),
\[
 \overline Z_i\longmapsto c_iE_i,\qquad c_i\in F^\times.
\]
The assignment \(E_i\mapsto c_iE_i\) is an automorphism.
\end{lemma}
\begin{proof}
The highest monomial of \(Z_i\) has coefficient one, so its reduction
is nonzero. The fundamental simple-root class has degree \(\alpha_i\)
in the graded identification of \cite[Theorem 6.1 and \S7.1]{HL}.
The \(\alpha_i\)-component of the positive part is one-dimensional.
Hence its image is a nonzero scalar multiple of \(E_i\). Every
presentation relation is root-homogeneous, so scaling the generators by
the nonzero \(c_i\) preserves all relations and has the inverse scaling.
\end{proof}

\begin{lemma}[Kernel descent to the local lattice]\label{lem:kerneldescent}
The local target \(\GG_{\OO}(Q)\) is torsion-free over the DVR \(\OO\).
If the generic homomorphism
\[
\Phi_K:U_{\mathbf q,\OO}^+\otimes_{\OO}K
\longrightarrow
\GG_{\OO}(Q)\otimes_{\OO}K
\]
has nonzero kernel, then \(\Phi_{\OO}\) has a nonzero kernel element
whose reduction modulo \(u-1\) is nonzero.
\end{lemma}
\begin{proof}
The local standard basis established earlier makes
\(\GG_{\OO}(Q)\) a free, hence torsion-free, \(\OO\)-module.
Let \(0\ne x_K\in\ker\Phi_K\).  Clearing the finitely many denominators
of its PBW coefficients gives \(0\ne x\in U_{\mathbf q,\OO}^+\) whose
image becomes zero after localization.  Therefore some nonzero
\(a\in\OO\) satisfies
\[
a\,\Phi_{\OO}(x)=0.
\]
Torsion-freeness of the target implies \(\Phi_{\OO}(x)=0\).

Write
\[
x=\sum_{\mathbf a}f_{\mathbf a}(u)E^{\mathbf a}
\]
in the free PBW basis and let
\[
d=\min_{f_{\mathbf a}\ne0}\operatorname{ord}_{u-1}f_{\mathbf a}.
\]
Then \(x=(u-1)^dy\) for a PBW-lattice element \(y\), and at least one
coefficient of \(y\) is a unit.  Since
\[
0=\Phi_{\OO}(x)=(u-1)^d\Phi_{\OO}(y)
\]
and the target is torsion-free, \(\Phi_{\OO}(y)=0\).  The reduction of
\(y\) modulo \(u-1\) is nonzero by construction.
\end{proof}

\begin{theorem}[Generic realization]\label{thm:main}
For every finite simply-laced bipartite quiver \(Q\),
\begin{equation}\label{eq:mainiso}
 \Phi:U_{\mathbf q}^+(\mathfrak g)\xrightarrow{\ \sim\ }
              \GG_R(Q)\otimes_RK,\qquad E_i\longmapsto Z_i.
\end{equation}
The parameters are given by \eqref{eq:intro-q}. Thus the isomorphism is
over \(K\), without the auxiliary scalar extension \(K'\).
\end{theorem}

\begin{remark}[Why the surjectivity argument is noncircular]
The equality of graded dimensions used at the end of
Theorem~\ref{thm:main} was established before the local spectral theorem,
via the Euler-correction identity, the graded torus twist, character
comparison, and the Hernandez--Leclerc one-parameter PBW basis.  It does
not use the multiparameter realization.  Thus injectivity plus equal
graded dimensions is a genuine final step rather than a restatement of
the theorem being proved.
\end{remark}

\begin{proof}
By \cref{prop:Cartan-rel,prop:PBW}, the assignment defines an
\(\OO\)-homomorphism
\[
 \Phi_{\OO}:U_{\mathbf q,\OO}^+\longrightarrow\GG_{\OO}(Q)
\]
and hence the desired generic homomorphism. Its special fiber is an
isomorphism by \cref{prop:localbasis,lem:simplefiber}.

If the generic kernel were nonzero, Lemma~\ref{lem:kerneldescent}
would produce a nonzero element of the local kernel with nonzero reduction
modulo \(u-1\).  But the special fiber of the source is
\(U_{t_0}^+(\mathfrak g)\), the special fiber of the target is the
Hernandez--Leclerc algebra, and by Lemma~\ref{lem:simplefiber} the induced
map is the diagonal automorphism
\[
E_i\longmapsto c_iE_i,\qquad c_i\in F^\times.
\]
Its kernel is zero, a contradiction.  Hence \(\Phi\) is injective.

The map preserves \(\LL_+\)-degrees. In every degree its source has
PBW dimension \(p(\gamma)\), while its target has the same dimension
by \cref{thm:graded-twist}. An injective map between these two finite
dimensional spaces is surjective. Therefore \(\Phi\) is an
isomorphism in every degree and proves \eqref{eq:mainiso}.
\end{proof}
Here specialization is taken on the local forms constructed above; it is
not evaluation of arbitrary elements of the generic field at \(u=1\).

\section{Homogeneous quantum seeds and the second form}\label{sec:cluster}
The Euler correction also has a direct cluster-theoretic meaning. We state
it for homogeneous seed data, separating compatibility from any additional
claim about categorical realizations of all cluster variables.

Let \((M,\widetilde B,\Lambda_0)\) be a based quantum seed with
\(N\) variables and \(m\) mutable directions, in the sense of
\cite{BZ}. Write \(g_1,\ldots,g_N\in\LL\) for its degrees and
let \(G\) be the matrix of these column vectors in the simple-root basis.
Homogeneity of the exchange relations is the condition
\begin{equation}\label{eq:GB}
 G\widetilde B=0.
\end{equation}
Indeed the two exchange monomials have degrees differing by
\(\sum_i b_{ik}g_i\). Let \(W_Q\) be the matrix of \(\om\).
Assume the seed variables are homogeneous elements in the graded torus
comparison of \cref{prop:torus-twist}.

\begin{proposition}[Compatibility of the Euler correction]\label{prop:cluster}
If
\[
 \widetilde B^T\Lambda_0=(2I_m\mid0),\qquad G\widetilde B=0,
\]
then the second toroidal commutation matrix is
\begin{equation}\label{eq:Lambda2}
 \Lambda_2=\frac12\left(\Lambda_0+G^TW_QG\right),
\end{equation}
and
\begin{equation}\label{eq:twocompat}
 \widetilde B^T\Lambda_2=(I_m\mid0).
\end{equation}
When \(\Lambda_2\) is integral, these are two compatible skew forms
for a toroidal seed with parameters \(t_0,t\).
\end{proposition}
\begin{proof}
The quasi-commutation factor of two homogeneous seed variables in the
moving one-parameter torus is \(v^{\Lambda_0(i,j)}\). Twisting by
\(\kappa\) changes it by
\(\kappa(g_i,g_j)/\kappa(g_j,g_i)=u^{\om(g_i,g_j)}\).
Thus the factor is
\[
 t_0^{\Lambda_0(i,j)}u^{\Lambda_0(i,j)+\om(g_i,g_j)}
 =t_0^{\Lambda_0(i,j)}t^{\Lambda_2(i,j)},
\]
proving \eqref{eq:Lambda2}. Matrix multiplication gives
\[
 \widetilde B^TG^TW_QG=(G\widetilde B)^TW_QG=0.
\]
Now \eqref{eq:twocompat} follows by taking one half of the first
compatibility identity. Skewness is clear. For a seed realized in the
original integral two-parameter based torus the integrality is already
part of its exponent data; for abstract homogeneous seed data it is an
explicit hypothesis, not an automatic consequence of homogeneity alone.
\end{proof}

\begin{proposition}[Mutation covariance]\label{prop:mutation}
Under homogeneous seed mutation, \eqref{eq:Lambda2} transforms by the
usual skew-form mutation. In particular, the two compatible forms remain
compatible along the mutation pattern.
\end{proposition}
\begin{proof}
Use the standard mutation matrix \(E\) on the lattice of seed exponents.
It is the identity on columns other than the mutated column \(k\), with
\[
 E e_k=-e_k+\sum_i[-b_{ik}]_+e_i.
\]
The degree of the new cluster variable is the common degree of the two
exchange terms minus \(g_k\); hence \(G'=GE\). The ordinary skew
form transforms as \(\Lambda_0'=E^T\Lambda_0E\). Therefore
\begin{align*}
 \Lambda_2'&=\tfrac12\bigl(\Lambda_0'+(G')^TW_QG'\bigr)\\
 &=E^T\tfrac12\bigl(\Lambda_0+G^TW_QG\bigr)E
   =E^T\Lambda_2E.
\end{align*}
The standard compatible-pair mutation identity of \cite{BZ} applies to
each form, or equivalently one may use \cite[Proposition 2.3]{FH}.
The mutated exchange relation remains homogeneous, so \(G'\widetilde
B'=0\) as required for the next step.
\end{proof}
This gives a verifiable seed-level statement and explains the constants
\(2\) and \(1\). It does not identify every toroidal cluster variable
with a deformed simple representation, and no such claim is used in
\cref{thm:main}. The distinction is important when comparing the present
algebraic realization with the wider cluster-structure expectation in
\cite[\S8.3]{FH}.

\section{Branching examples and exact kernel checks}\label{sec:examples}
The following examples compute root labels and inverse-Cartan data. They
are not presented as independent enumerations of all monomials of an
exceptional fundamental character. The spectrum statements follow from the
proved structural results.

\subsection{Type \texorpdfstring{\(D_4\)}{D4}}
Use the labeling with central vertex \(2\) and sources \(I_+=\{2\}\):
\begin{figure}[htbp]
\centering
\begin{tikzpicture}[>=Stealth,
  vertex/.style={circle,draw,inner sep=1.8pt,minimum size=7mm}]
\node[vertex] (c) at (0,0) {$2$};
\node[vertex] (l) at (-1.7,0) {$1$};
\node[vertex] (r) at (1.7,0) {$3$};
\node[vertex] (b) at (0,-1.35) {$4$};
\draw[->] (c)--(l);
\draw[->] (c)--(r);
\draw[->] (c)--(b);
\end{tikzpicture}
\caption{The bipartite orientation of \(D_4\) used in the branching
calculation.  The trivalent vertex \(2\) is the unique source.}
\label{fig:D4quiver}
\end{figure}
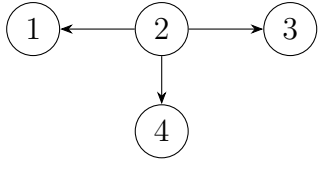
Thus \(\xi_2=1\), \(\xi_1=\xi_3=\xi_4=0\), and
\[
 \tau=s_2s_1s_3s_4,\qquad
 \gamma_2=\alpha_2,\qquad \gamma_j=\alpha_j+\alpha_2\quad(j=1,3,4).
\]
The complete level-zero root labeling is given in \cref{tab:D4}. It can
be checked by applying \(\tau\) once for each downward shift of two.
\begin{table}[htbp]
\centering
\caption{The twelve positive-root positions for the chosen \(D_4\) quiver.}
\label{tab:D4}
\begin{tabular}{ccl@{\qquad}ccl}
\toprule
\(i\)&\(p\)&\(\beta_{i,p}\)&\(i\)&\(p\)&\(\beta_{i,p}\)\\
\midrule
1&0&\(\alpha_1+\alpha_2\)&2&1&\(\alpha_2\)\\
1&\(-2\)&\(\alpha_2+\alpha_3+\alpha_4\)&2&\(-1\)&\(\alpha_1+2\alpha_2+\alpha_3+\alpha_4\)\\
1&\(-4\)&\(\alpha_1\)&2&\(-3\)&\(\alpha_1+\alpha_2+\alpha_3+\alpha_4\)\\
3&0&\(\alpha_2+\alpha_3\)&4&0&\(\alpha_2+\alpha_4\)\\
3&\(-2\)&\(\alpha_1+\alpha_2+\alpha_4\)&4&\(-2\)&\(\alpha_1+\alpha_2+\alpha_3\)\\
3&\(-4\)&\(\alpha_3\)&4&\(-4\)&\(\alpha_4\)\\
\bottomrule
\end{tabular}
\end{table}
In particular the simple-root positions are
\begin{equation}\label{eq:D4simple}
 x_1=(1,-4),\quad x_2=(2,1),\quad x_3=(3,-4),\quad x_4=(4,-4).
\end{equation}
The source class is \(Z_2=Y_{2,1}\), while the other distinguished
classes have highest monomials \(Y_{j,-4}\).

Let \(H\) be the adjacency matrix. The inverse-Cartan recurrence gives
\[
 \begin{aligned}
 \ct(1)&=I,\qquad \ct(2)=H,\qquad \ct(3)=H^2-I,\\
 \ct(4)&=H^3-2H,\qquad \ct(6)=H^5-4H^3+3H.
 \end{aligned}
\]
For a central--leaf entry, \((H)_{2j}=1\),
\((H^3)_{2j}=3\), and \((H^5)_{2j}=9\), so
\[
 \ct_{2j}(4)=1,\qquad \ct_{2j}(6)=0\qquad(j=1,3,4).
\]
The spectral gap between \((2,1)\) and \((j,-4)\) is \(5\).
Consequently
\[
 D_0(Y_{2,1},Y_{j,-4})=\ct_{2j}(4)-\ct_{2j}(6)=1,
 \qquad D^{(2)}(Y_{2,1},Y_{j,-4})=\ct_{2j}(4)=1.
\]
This verifies the base pair independently of the Euler-form evaluation.
At the branch the root monomial is
\[
 A_{2,r}=Y_{2,r-1}Y_{2,r+1}Y_{1,r}^{-1}Y_{3,r}^{-1}Y_{4,r}^{-1}.
\]
Despite the three neighbors, only \(A_{2,0}\) and \(A_{2,2}\) pair
nontrivially with \(Y_{2,1}\). The spectral argument gives the same
pair of eigenvalues on each of the three edges.

The parameter matrix is particularly transparent:
\[
 \mathbf q=\begin{pmatrix}
 t_0^2t&t_0^{-1}t^{-1}&1&1\\
 t_0^{-1}&t_0^2t&t_0^{-1}&t_0^{-1}\\
 1&t_0^{-1}t^{-1}&t_0^2t&1\\
 1&t_0^{-1}t^{-1}&1&t_0^2t
 \end{pmatrix}.
\]
The leaf--leaf parameters are exactly \(1\), not unspecified reciprocal
units. This is an explicit consequence of the Euler correction identity.

\subsection{Type \texorpdfstring{\(E_6\)}{E6} with the standard trivalent vertex}
Use the chain \(1-3-4-5-6\), with vertex \(2\) attached to vertex
\(4\).  Orient from sources \(\{1,4,6\}\) to sinks \(\{2,3,5\}\), as in
\cref{fig:E6quiver}.
\begin{figure}[htbp]
\centering
\begin{tikzpicture}[>=Stealth,
  vertex/.style={circle,draw,inner sep=1.8pt,minimum size=7mm}]
\node[vertex] (n1) at (0,0) {$1$};
\node[vertex] (n3) at (1.45,0) {$3$};
\node[vertex] (n4) at (2.90,0) {$4$};
\node[vertex] (n5) at (4.35,0) {$5$};
\node[vertex] (n6) at (5.80,0) {$6$};
\node[vertex] (n2) at (2.90,1.25) {$2$};
\draw[->] (n1)--(n3);
\draw[->] (n4)--(n3);
\draw[->] (n4)--(n5);
\draw[->] (n6)--(n5);
\draw[->] (n4)--(n2);
\end{tikzpicture}
\caption{The bipartite orientation of \(E_6\) used in the exceptional
branching check.  The trivalent vertex \(4\) is a source.}
\label{fig:E6quiver}
\end{figure}
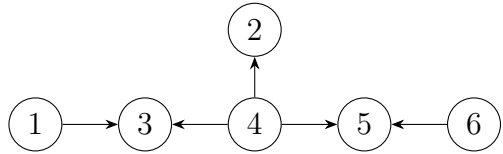
The three arms at vertex \(4\) have lengths \(2,2,1\). In
particular this is the \(E_6\) diagram, not the length pattern
\(1,1,3\) of \(D_6\). Here
\[
 \tau=s_1s_4s_6s_2s_3s_5,
\]
and the initial roots are
\begin{align*}
 \gamma_1&=\alpha_1,&\gamma_4&=\alpha_4,&\gamma_6&=\alpha_6,\\
 \gamma_2&=\alpha_2+\alpha_4,&
 \gamma_3&=\alpha_1+\alpha_3+\alpha_4,&
 \gamma_5&=\alpha_4+\alpha_5+\alpha_6.
\end{align*}
The downward Coxeter recursion yields the following simple-root positions:
\begin{equation}\label{eq:E6simple}
 \begin{array}{c|rrrrrr}
 \text{simple root }i&1&2&3&4&5&6\\ \hline
 \kappa_i&1&2&5&4&3&6\\
 p_i&1&-10&-10&1&-10&1
 \end{array}
\end{equation}
For example, \(\tau^5\gamma_5=\alpha_3\), and
\(\tau^5\gamma_3=\alpha_5\). This illustrates the need to keep
\(i\), the simple-root label, distinct from \(\kappa_i\), the first
index of the corresponding \(Y\)-variable.

Consider the edge \(4\to3\). Its source is \(Y_{4,1}\), while the
sink class has highest monomial \(Y_{5,-10}\). The recurrence gives
\[
 \ct_{45}(10)=1,\qquad\ct_{45}(12)=0.
\]
Thus the spectral gap \(11\) again gives the pair \((1,1)\).
The local branch monomial is
\[
 A_{4,r}=Y_{4,r-1}Y_{4,r+1}Y_{2,r}^{-1}Y_{3,r}^{-1}Y_{5,r}^{-1}.
\]
The only nonzero source wall pairings occur at its centers \(0\) and
\(2\). \Cref{thm:two} determines the resulting spectrum without an
explicit expansion of the whole fundamental character.

\subsection{Reproducible finite checks}
The supplementary program \texttt{verify\_ade\_kernels.py} uses only
integer arithmetic and the Python standard library. It constructs the
Cartan matrices, the simple reflections, the Coxeter element, and the
inverse-Cartan recurrence independently, and then checks
\eqref{eq:D2closed} and \eqref{eq:kernel-main} on the window
\[
 (i,p)=(i,\xi_i-2k),\qquad -12\le k\le12.
\]
All ordered pairs in this window are checked. In addition, the program
walks down each positive Coxeter column to identify the simple-root points
and checks \eqref{eq:base-pair} on every source--sink edge.
\begin{table}[htbp]
\centering
\caption{Exact kernel checks. The positive-root counts are obtained by the
Coxeter-column computation and checked against the finite-type counts.}
\label{tab:checks}
\begin{tabular}{lrr@{\hspace{1.5em}}lrr}
\toprule
Type&\shortstack{Positive\\roots}&\shortstack{Ordered\\pairs}&
Type&\shortstack{Positive\\roots}&\shortstack{Ordered\\pairs}\\
\midrule
\(A_2\)&3&2,500&\(D_5\)&20&15,625\\
\(A_3\)&6&5,625&\(D_6\)&30&22,500\\
\(A_4\)&10&10,000&\(E_6\)&36&22,500\\
\(A_6\)&21&22,500&\(E_7\)&63&30,625\\
\(D_4\)&12&10,000&\(E_8\)&120&40,000\\
\bottomrule
\end{tabular}
\end{table}
The total is \(181{,}875\) ordered kernel pairs, all satisfying the two
identities. These finite tests check conventions and computations; they
do not replace \cref{thm:kernel}, do not establish an identity for
unbounded spectral positions by themselves, and do not enumerate the
monomial support of an \(E_6,E_7\), or \(E_8\) character. The JSON
output records the quiver orientations and all computed simple-root
positions, so the tests can be reproduced without any representation-theory
software package.

\section{Concluding remarks}
The two-parameter deformation has a useful decomposition into a moving
one-parameter product and an alternating degree twist. The decisive
identity is not merely the local vanishing of most \(A\)--\(Y\)
pairings: it is the identification of their residual form on the spectral
quotient with the skew Euler form. This supplies the positive grading and
the homogeneous dimensions required by the spectral argument.

The quadratic-annihilator principle then explains why non-thin modules
and branching vertices do not introduce a third local eigenvalue. The
special fiber is used only after the generic spectral bound has been
established. Both Serre relations and the nonadjacent commutation factors
are consequently determined. The local coefficient ring makes the final
specialization argument precise, including the cancellation of the
uniformizer in a torsion-free target.

Several extensions require separate work. One may seek preferred global
integral forms and basis-compatible isomorphisms, rather than the local
and generic forms used here. The seed compatibility identity suggests
further cluster-theoretic applications, but categorical interpretations of
all toroidal cluster variables are not proved by compatibility alone.
For non-simply-laced types, the inverse-Cartan recurrence and the local
root factors have different shifts; the quotient calculation and the
rank-two spectral bound would both have to be modified. None of these
stronger statements is an assumption of the present realization.

\section*{Funding}
This work is supported in part by the National Natural Science
Foundation of China (No.~12271085). It is also supported by the Natural
Science Foundation of Heilongjiang Province (No.~PL2024A007).

\section*{Data availability statement}
Data sharing is not applicable to this article as no datasets were
generated or analysed during the current study.

\section*{Declaration of competing interests}
The authors declare that they have no known competing financial interests
or personal relationships that could have appeared to influence the work
reported in this paper.

\end{document}